\documentclass[11pt]{amsart}

\usepackage{mathrsfs,amsmath,amsfonts}
\usepackage{amsthm}
\usepackage{amssymb}
\usepackage{fancyhdr}
\usepackage{CJK}
\usepackage{arydshln}
\usepackage[all]{xy}
\usepackage{tikz}
\usepackage{cite}
\usepackage{hyperref}
\usepackage[margin=1in]{geometry}
\newtheorem{thm}{Theorem}[section]
\newtheorem{prop}[thm]{Proposition}
\newtheorem{defini}[thm]{Definition}

\newtheorem{remark}[thm]{Remark}

\title{Deformations and Cohomology Theory of Rota-Baxter Systems}
\author{Yuming Liu, Kai Wang, Liwen Yin*}

\address{Yuming Liu and LiWen Yin
\newline School of Mathematical Sciences
\newline Laboratory of Mathematics and Complex Systems
\newline Beijing Normal University
\newline Beijing 100875
\newline P.R.China}
\email{ymliu@bnu.edu.cn}
\email{lwyinbnu@qq.com}
\address{Kai Wang
\newline School of Mathematical Sciences
\newline Shanghai Key Laboratory of PMMP
\newline East China Normal University
\newline Shanghai 200241
\newline P.R.China}
\email{wangkai@math.ecnu.edu.cn}

\date{\today}

\begin{document}
\renewcommand{\thefootnote}{\alph{footnote}}
\setcounter{footnote}{-1} \footnote{* Corresponding author.}
\setcounter{footnote}{-1} 
\renewcommand{\thefootnote}{\alph{footnote}}
\setcounter{footnote}{-1} 

\maketitle

\begin{abstract}
Inspired by the work of Wang and Zhou \cite{WZ} for Rota-Baxter algebras, we develop a cohomology theory of Rota-Baxter systems and justify it by interpreting the lower degree cohomology groups as formal deformations and as abelian extensions of Rota-Baxter systems. A further study on an $L_\infty$-algebra structure associated to this cohomology theory will be given in a subsequent paper.
\end{abstract}	

\tableofcontents
\section{Introduction}
By a general philosophy, the deformation theory of any given mathematical object can be described by a certain differential graded (=dg) Lie algebra or more generally a $L_\infty$-algebra associated to the mathematical object. Therefore it is an important question to construct explicitly this dg Lie algebra
or $L_\infty$-algebra governing deformation theory of this mathematical object. Another important question about algebraic structures is to study their homotopy versions, just like $A_\infty$-algebras for usual associative algebras. The nicest result would be providing a minimal model of the operad governing an algebraic structure. When this operad is Koszul, there exists a general theory, the so-called Koszul duality for operads, which defines a homotopy version of this algebraic structure via the cobar construction of the Koszul dual cooperad, which, in this case, is a minimal model. However, when the operad is NOT Koszul, essential difficulties arise and few examples of minimal models have been worked out. These two questions, say, describing controlling $L_\infty$-algebras and constructing homotopy versions, are closed related. In fact, given a cofibrant resolution, in particular, a minimal model, of
the operad in question, one can form the deformation complex of the algebraic structure and construct its $L_\infty$-structure, see \cite[Introduction]{WZ} for more explanation on this method.

Recently Wang and Zhou developed an ad hoc method in \cite{WZ} to deal with the above two questions. Surprisingly, this method works well for many individual nonKoszul algebra structures. The idea of this method is as follows. Given an algebraic structure on a space $V$ realised as an algebra over an operad, by considering the formal deformations of this algebraic structure, we first construct the deformation complex and find an $L_\infty$-structure on the underlying graded space of this complex such that the Maurer-Cartan elements are in bijection with the algebraic structures on $V$. When $V$ is graded, we define a homotopy version of this algebraic structure as Maurer-Cartan elements in the $L_\infty$-algebra constructed above. Finally under suitable conditions, we could show that the operad governing the homotopy version is a minimal model of the original operad. Using their method, Wang and Zhou successfully find a minimal model of the operad governing Rota-Baxter algebra structures in \cite{WZ}. Recall that an associative algebra $A$ over a field $\mathbb{K}$ is called a Rota-Baxter algebra of weight $\lambda$ (where $\lambda\in \mathbb{K}$) if there is a $\mathbb{K}$-linear operator $R: A\to A$ (called a Rota-Baxter operator of weight $\lambda$) satisfying the equation $R(a)R(b)=R(R(a)b+aR(b)+\lambda ab)$ for all $a,b\in A$.

In the present paper, we discuss the formal deformations and cohomology theory of Rota-Baxter systems. Rota-Baxter system is a generalisation of the notion of Rota–Baxter algebra introduced by Brzezi\'{n}ski \cite{B}. This generalisation consists of two operators $R,S$ acting on an associative algebra $A$ and satisfying equations similar to the Rota–Baxter equation. Since the two operators $R,S$ are wrapped up in each other, in order to develop a cohomology theory of Rota-Baxter systems, we have to modify the construction of the cohomology cochain complex in \cite{WZ} for Rota-Baxter algebras, see Section 4 for the details.

In a subsequent paper, we will discuss an $L_\infty$-algebra structure over the cochain complex of the Rota-Baxter system and realize the Rota-Baxter system structures as the Maurer-Cartan elements of this $L_\infty$-algebra, parallel to the work on Rota-Baxter algebras in \cite[Section 8]{WZ}.

This paper is organised as follows. Section 2 contains a quick review on formal deformations and Hochschild cohomology of associative algebras. In Section 3 we give basic definitions and facts about Rota-Baxter systems and Rota-Baxter system bimodules. A cohomology cochain complex of Rota-Baxter system operators, and with the help of the usual Hochschild cocohain complex, a cochain complex, whose cohomology groups should control deformation theory of
Rota-Baxter systems, is exhibited in Section 4. We justify this cohomology theory by interpreting lower degree cohomology groups as formal deformations (Section 5) and as abelian extensions of Rota-Baxter systems (Section 6).

\section{A quick review on formal deformations and Hochschild cohomology of associative algebras}
Throughout this paper, let $\mathbb{K}$ be a field of arbitrary characteristic. All vector spaces are defined over $\mathbb{K}$, all tensor products and Hom-spaces are taken over $\mathbb{K}$, unless otherwise stated. Besides, all algebras considered in this paper are associative (but not necessarily unital) over $\mathbb{K}$. We say that an algebra $A$ is non-degenerate provided that for any $b\in A$, $ba=0$ or $ab=0$ for all $a\in A$ implies that $b=0$. Obviously, any unital algebra is non-degenerate. We denote by $\mathbb{K}[[t]]$ the power series ring in one variable $t$ over the field $\mathbb{K}$.

In this section, we give a quick review on formal deformations and Hochschild cohomology of associative algebras. For backgrounds and more details on these subjects, we refer to \cite{G1}, \cite{G2}, and \cite{WZ}.

\subsection{Hochschild cohomology of associative algebras}
Let $(A, \mu)$ be an associative $\mathbb{K}$-algebra. We often write $\mu(a\otimes b) = a \cdot b = ab$ for any $a, b \in A$, and for $a_1,\dots, a_n\in A$, we write $a_{1,n}:=a_1\otimes \cdots \otimes a_n \in A^{\otimes n}$. Let $M$ be a bimodule over $A$. Recall that the Hochschild cochain complex of $A$ with coefficients in $M$ is
\[\mathrm{C}_{\mathrm{Alg}}^\bullet(A, M) := \bigoplus_{n=0}^\infty \mathrm{C}_{\mathrm{Alg}}^n(A, M),\]
where $\mathrm{C}_{\mathrm{Alg}}^n(A, M) = \mathrm{Hom}(A^{\otimes n}, M)$ and the differential $\delta^n : \mathrm{C}_{\mathrm{Alg}}^n(A, M) \to \mathrm{C}_{\mathrm{Alg}}^{n+1}(A, M)$ is defined as:
\begin{align*}
& \delta^n(f)(a_{1,n+1}) \\
={} & (-1)^{n+1}a_1f(a_{2,n+1}) + \sum_{i=1}^n(-1)^{n-i+1}f(a_{1,i-1}\otimes a_i\cdot a_{i+1}\otimes a_{i+2,n+1}) + f(a_{1,n})a_{n+1}
\end{align*}
for all $f\in \mathrm{C}_{\mathrm{Alg}}^n(A, M), a_1,\ldots ,a_{n+1}\in A.$

The cohomology of the Hochschild cochain complex $\mathrm{C}_{\mathrm{Alg}}^\bullet(A, M)$ is called the Hochschild cohomology of $A$ with coefficients in $M$, denoted by $\mathrm{HH}^\bullet(A,M)$. When the bimodule $M$ is the regular bimodule $A$ itself, we just denote $\mathrm{C}_{\mathrm{Alg}}^\bullet(A, A)$ by $\mathrm{C}_{\mathrm{Alg}}^\bullet(A)$ and call it the Hochschild cochain complex of the associative algebra $(A, \mu)$. Denote the cohomology $\mathrm{HH}^\bullet(A, A)$ by $\mathrm{HH}^\bullet(A)$, called the Hochschild cohomology of the associative algebra $(A, \mu)$.

\subsection{Formal deformations of associative algebras}
Given an associative $\mathbb{K}$-algebra $(A, \mu)$, consider $\mathbb{K}[[t]]$-bilinear associative multiplications on
\[ A[[t]]=\{\sum_{i=0}^\infty a_it^i| a_i\in A, \forall i\geqslant 0 \}.\]
Such a multiplication is determined by
\[\mu_t = \sum_{i=0}^\infty \mu_it^i: A\otimes A \to A[[t]],\]
where for all $i\geqslant 0,\,\mu_i : A\otimes A \to A$ are $k$-linear maps. When $\mu_0 = \mu$, we say that $\mu_t$
is a formal deformation of $\mu$ and $\mu_1$ is called the infinitesimal of the formal deformation $\mu_t$.
The only constraint is the associativity of $\mu_t$:
\[ \mu_t(a\otimes \mu_t(b\otimes c)) = \mu_t(\mu_t(a\otimes b)\otimes c), \forall a, b, c\in A,\]
which is equivalent to the following family of equations:
\begin{equation}
  \sum_{i+j=n} \mu_i(\mu_j(a\otimes b)\otimes c)-\mu_i(a\otimes \mu_j(b\otimes c))=0,\forall a, b, c\in A, n\geqslant 0.\label{2}
\end{equation}
Looking closely at the cases $n = 0$ and $n = 1$, one obtains:

  (i) when $n = 0, (a \cdot b) \cdot c = a \cdot (b \cdot c), \forall a, b, c \in A$, which is exactly the associativity of $\mu$;

 (ii) when $n=1$,
  \[a\mu_1(b\otimes c)-\mu_1(ab\otimes c)+\mu_1(a\otimes bc)-\mu_1(a\otimes b)c=0, \forall a, b, c \in A,\]
  which says that the infinitesimal $\mu_1$ is a 2-cocycle in the Hochschild cochain complex $\mathrm{C}_{\mathrm{Alg}}^\bullet(A)$.

\section{Rota-Baxter systems and Rota-Baxter system bimodules}
In this section, we first recall the definition of Rota-Baxter systems and define the Rota-Baxter system bimodules, and then we give several interesting observations about them, following similar ideas from \cite[Definiton 4.3]{WZ}.

\begin{defini} \label{definition-RBS} {\rm(see \cite[Definition 2.1]{B})} A triple $(A,R,S)$ consisting of an associative algebra $A$ and two $\mathbb{K}$-linear operators $ R, S : A\to A$ is called a Rota-Baxter system if , for all $a,b\in A$,
\begin{gather}
R(a)R(b)=R(R(a)b+aS(b)),\label{3}\\
S(a)S(b)=S(R(a)b+aS(b)).\label{4}
\end{gather}
In this case, $(R, S)$ are called Rota-Baxter system operators.
\end{defini}

\begin{remark} {\rm(see \cite[Lemma 2.2]{B})} \label{R+id} Let $A$ be an algebra. If $R$ is a Rota-Baxter operator of weight $\lambda$ on $A$, then $(A, R, R + \lambda id)$ and $(A, R + \lambda id, R)$ are Rota-Baxter systems.
\end{remark}

\begin{remark} {\rm(see \cite[Lemma 2.3]{B})} Let $A$ be an algebra. Let $R:A\to A$ be a left $A$-linear map and $S:A\to A$  be a right $A$-linear map.
Then $(A, R, S)$ is a  Rota-Baxter system if and only if, for all $a,b\in A$,
\begin{equation}
  aR\circ S(b)=0=S\circ R(a)b.
\end{equation}
In particular, if $A$ is a non-degenerate algebra, then $(A, R, S)$ is a  Rota-Baxter system if and only if $R$ and $S$ satisfy the orthogonality condition
\begin{equation}
  R\circ S=S\circ R=0.
\end{equation}
\end{remark}

\begin{remark} {\rm(see \cite[Remark 2.6]{B})} \label{morphism-RBS} A morphism of Rota-Baxter systems from $(A, R_A, S_A)$ to $(B, R_B, S_B)$ is an algebra map $ f : A\to B$ rendering the following diagrams commutative:
$$\xymatrix{
 A \ar[r]^{R_A}\ar[d]_f & A \ar[d]^f  &A \ar[r]^{S_A}\ar[d]_f & A \ar[d]^f\\
 B\ar[r]^{R_B} & B,  & B\ar[r]^{S_B} & B.
 }$$
 \end{remark}

\begin{defini} Let $(A, R, S)$ be a Rota-Baxter system and $M$ be a bimodule over the associative algebra $A$. We say that $M$ is a bimodule over Rota-Baxter system $(A, R, S)$ or a Rota-Baxter system bimodule if $M$ is endowed with two linear operators $ R_M, S_M : M \to M $ such that the following equations hold for any $a \in A$ and $m \in M$:
\begin{gather}
R(a)R_M(m)=R_M(R(a)m+aS_M(m)),\label{5}\\
R_M(m)R(a)=R_M(R_M(m)a+mS(a)),\\
S(a)S_M(m)=S_M(R(a)m+aS_M(m)),\\
S_M(m)S(a)=S_M(R_M(m)a+mS(a)).
\end{gather}
\end{defini}

Of course, $(A, R, S)$ itself is a bimodule over the Rota-Baxter system $(A, R, S)$, called the regular Rota-Baxter system bimodule.

\begin{prop}   {\rm(Compare to \cite[Proposition 4.4]{WZ})} Let $(A, R, S)$ be a Rota-Baxter system and $M$ be a bimodule over associative algebra $A$. It is well known that $A\oplus M$ becomes an associative algebra whose multiplication is
\begin{equation}
(a, m)(b, n) = (ab, an + mb).
\end{equation}
Write $\iota : A\to {A\oplus M}, a\mapsto (a, 0)$ and $\pi : {A\oplus M} \to A, (a, m) \mapsto a$. Then ${A\oplus M}$ is a Rota-Baxter system such that $\iota$ and $\pi$ are both morphisms of Rota-Baxter systems if and only if $M$ is a Rota-Baxter system bimodule over $A$.
This new Rota-Baxter system will be denoted by $A\ltimes M$, called the semi-direct product (or trivial extension) of $A$ by $M$.
\end{prop}

\begin{proof}
First we notice that the associative algebra ${A\oplus M}$ is a Rota-Baxter system such that $\iota$ and $\pi$ are both morphisms of Rota-Baxter systems
if and only if there exist two $\mathbb{K}$-linear operators $ R^{\prime}, S^{\prime} : A\oplus M \to A\oplus M$ such that the following equations hold for any $a_1, a_2, a \in A$ and $m_1, m_2, m \in M$:
\begin{gather}
  R^{\prime}(a_1,m_1) R^{\prime}(a_2,m_2)=R^{\prime}\Bigl(R^{\prime}(a_1,m_1)(a_2,m_2) +(a_1,m_1) S^{\prime}(a_2,m_2)\Bigr),\label{R} \\
  S^{\prime}(a_1,m_1) S^{\prime}(a_2,m_2)=S^{\prime}\Bigl(R^{\prime}(a_1,m_1)(a_2,m_2) +(a_1,m_1) S^{\prime}(a_2,m_2)\Bigr),\label{S} \\
  R^{\prime}\circ \iota (a)=\iota \circ R(a)=(R(a),0),\label{i} \\
  \pi \circ R^{\prime}(0,m)=R\circ \pi (0,m)=0,\label{ii}\\
   S^{\prime}\circ \iota (a)=\iota \circ S(a)=(S(a),0), \label{iii} \\
  \pi \circ S^{\prime}(0,m)=S\circ \pi (0,m)=0, \label{iv}
  \end{gather}
where the last four equations (\ref{i})-(\ref{iv}) come from the definition of morphism of Rota-Baxter systems (cf. Remark \ref{morphism-RBS}) and are equivalent to the following equations
\begin{gather*}
  R^{\prime}(a,0)=(R(a),0),\\
  R^{\prime}(0,m)=(0,R_2^{\prime}(0,m)),\\
  S^{\prime}(a,0)=(S(a),0),\\
  S^{\prime}(0,m)=(0,S_2^{\prime}(0,m)).
\end{gather*}

Now we assume that the latter part of the above equivalent conditions holds and write $R_M(m):=R_2^{\prime}(0,m)$ and $S_M(m):=S_2^{\prime}(0,m)$.
Then $ R_M, S_M : M \to M $ become two linear operators, and we have
\[R^{\prime}(a,m)=R^{\prime}(a,0)+R^{\prime}(0,m)=(R(a),0)+(0,R_M(m))=(R(a),R_M(m)),\]
\[S^{\prime}(a,m)=S^{\prime}(a,0)+S^{\prime}(0,m)=(S(a),0)+(0,S_M(m))=(S(a),S_M(m)).\]
Furthermore, Equation (\ref{R}) becomes
\begin{align*}
  & \Bigl(R(a_1)R(a_2), R(a_1)R_M(m_2)+R_M(m_1)R(a_2)\Bigr)  \\
  ={} &\Bigl(R(a_1),R_M(m_1)\Bigr)\Bigl(R(a_2),R_M(m_2)\Bigr)\\
  ={} & R^{\prime}\Bigl(\bigl(R(a_1),R_M(m_1)\bigr)(a_2,m_2)+(a_1,m_1)\bigl(S(a_2),S_M(m_2)\bigr)\Bigr) \\
  ={} & R^{\prime}\Bigl(\bigl(R(a_1)a_2,R(a_1)m_2+R_M(m_1)a_2\bigr)+\bigl(a_1S(a_2),a_1S_M(m_2)+m_1S(a_2)\bigr)\Bigr)\\
  ={} & R^{\prime}\Bigl(R(a_1)a_2+a_1S(a_2),\bigl(R(a_1)m_2+a_1S_M(m_2)\bigr)+\bigl(R_M(m_1)a_2+m_1S(a_2)\bigr)\Bigr)\\
  ={} & \Bigl(R(a_1)R(a_2),R_M\bigl(R(a_1)m_2+a_1S_M(m_2)\bigr)+R_M\bigl(R_M(m_1)a_2+m_1S(a_2)\bigr)\Bigr),
    \end{align*}
and Equation (\ref{S}) becomes
\begin{align*}
  & \Bigl(S(a_1)S(a_2), S(a_1)S_M(m_2)+S_M(m_1)S(a_2)\Bigr)  \\
  ={} & \Bigl(S(a_1),S_M(m_1)\Bigr)\Bigl(S(a_2),S_M(m_2)\Bigr)\\
  ={} & S^{\prime}\Bigl(\bigl(R(a_1),R_M(m_1)\bigr)(a_2,m_2)+(a_1,m_1)\bigl(S(a_2),S_M(m_2)\bigr)\Bigr)\\
   ={} & S^{\prime}\Bigl(\bigl(R(a_1)a_2,R(a_1)m_2+R_M(m_1)a_2\bigr)+\bigl(a_1S(a_2),a_1S_M(m_2)+m_1S(a_2)\bigr)\Bigr)\\
  ={} & S^{\prime}\Bigl(R(a_1)a_2+a_1S(a_2),\bigl(R(a_1)m_2+a_1S_M(m_2)\bigr)+\bigl(R_M(m_1)a_2+m_1S(a_2)\bigr)\Bigr)\\
  ={} & \Bigl(S(a_1)S(a_2),S_M\bigl(R(a_1)m_2+a_1S_M(m_2)\bigr)+S_M\bigl(R_M(m_1)a_2+m_1S(a_2)\bigr)\Bigr).
    \end{align*}
Let $m_1=0$, we get the following equations
\begin{gather*}
  R(a_1)R_M(m_2)=R_M\bigl(R(a_1)m_2+a_1S_M(m_2)\bigr),\\
  S(a_1)S_M(m_2)=S_M\bigl(R(a_1)m_2+a_1S_M(m_2)\bigr).
\end{gather*}
Similarly, let $m_2=0$, we get
\begin{gather*}
  R_M(m_1)R(a_2)=R_M\bigl(R_M(m_1)a_2+m_1S(a_2)\bigr), \\
  S_M(m_1)S(a_2)=S_M\bigl(R_M(m_1)a_2+m_1S(a_2)\bigr).
\end{gather*}
The above four equations show that $M$ is a Rota-Baxter system bimodule over $A$.

Finally, given a Rota-Baxter system bimodule $M$ over $A$ with two $\mathbb{K}$-linear operators $R_M, S_M$, we define two $\mathbb{K}$-linear operators $R',S': A\oplus M\rightarrow A\oplus M$ by
$$R'(a,m)=(R(a),R_M(m)), \quad S'(a,m)=(S(a),S_M(m)).$$
Then it is easy to verify that $R'$ and $S'$ satisfy Equations (\ref{R})-(\ref{iv}). This finishes the proof of our proposition.
\end{proof}

\begin{prop}\label{star} \rm{(see \cite[Corollary 2.7]{B})} Let $(A, \mu, R, S)$ be a Rota-Baxter system. Define a new binary operation over $A$ as
\begin{equation}
  a\star b:= R(a)\cdot b+ a\cdot S(b)
\end{equation}
for any $a,b\in A$. Then the operation $\star$ is associative and $(A,\star)$ is a new associative algebra and we denote it by $A_\star$.
\end{prop}

\begin{remark} \label{new-RBS} Let $(A, \mu, R, S)$ be a Rota-Baxter system. If $R$ and $S$ satisfy $R\circ S=S\circ R$,
then one can verify that $(A,\star, R, S)$ is also a Rota-Baxter system.
\end{remark}

One can construct new bimodules over the associative algebra $A_\star$ from Rota-Baxter system bimodules over $(A, \mu, R, S)$.

\begin{prop}\label{triangle} Let $(A, \mu, R, S)$ be a Rota-Baxter system and $(M, R_M, S_M)$ be a Rota-Baxter system bimodule over it. We define a left action $``\vartriangleright"$, and a right action $``\vartriangleleft"$ of $A$ on the space $M\oplus M$ as follows: for any $a\in A, m_1,m_2\in M$,
\begin{gather}
  a\vartriangleright (m_1,m_2):= (R(a)m_1-R_M(am_2),S(a)m_2-S_M(am_2)), \\
  (m_1,m_2)\vartriangleleft a:= (m_1R(a)-R_M(m_1a),m_2S(a)-S_M(m_1a)).
\end{gather}
Then these actions make $M\oplus M$ into a bimodule over $A_\star$, and we denote this new bimodule by $_\vartriangleright \mathscr{D}(M)_\vartriangleleft$.
\end{prop}
\begin{proof} Firstly, we show that $(M\oplus M, \vartriangleright)$ is a left module over $A_\star$, that is
\[a \vartriangleright (b\vartriangleright (m_1,m_2)) = (a\star b)\vartriangleright (m_1,m_2).\]
On the one hand,
\begin{align*}
  & a \vartriangleright (b\vartriangleright (m_1,m_2))\\
   ={} &  a \vartriangleright \Bigl((R(b)m_1-R_M(bm_2),S(b)m_2-S_M(bm_2))\Bigr)\\
  ={} & \Bigl(R(a)(R(b)m_1-R_M(bm_2))- R_M(aS(b)m_2-aS_M(bm_2)),\\
  & S(a)(S(b)m_2-S_M(bm_2))- S_M(aS(b)m_2-aS_M(bm_2))\Bigr)\\
  ={} & \Bigl(R(a)R(b)m_1-R_M(R(a)bm_2+aS(b)m_2),\\
  & S(a)S(b)m_2-S_M(R(a)bm_2+aS(b)m_2)\Bigr);
\end{align*}
On the other hand,
\begin{align*}
  & (a\star b)\vartriangleright (m_1,m_2) \\
  ={} & \Bigl(R(a\star b)m_1-R_M((a\star b)m_2),S(a\star b)m_2-S_M((a\star b)m_2)\Bigr) \\
  ={} & \Bigl(R(a)R(b)m_1-R_M(R(a)bm_2+aS(b)m_2),\\
  & S(a)S(b)m_2-S_M(R(a)bm_2+aS(b)m_2)\Bigr).
\end{align*}

Next one can similarly check that the operation $\vartriangleleft$ defines a right module structure on $M\oplus M$ over $A_\star$.

Finally, we have the
equations:
\begin{align*}
  & (a\vartriangleright (m_1,m_2))\vartriangleleft b \\
  ={} & \Bigl(R(a)m_1-R_M(am_2),S(a)m_2-S_M(am_2)\Bigr)\vartriangleleft b  \\
  ={} & \Bigl((R(a)m_1-R_M(am_2))R(b)-R_M(R(a)m_1b-R_M(am_2)b),\\
  & (S(a)m_2-S_M(am_2))S(b)-S_M(R(a)m_1b-R_M(am_2)b)\Bigr)\\
  ={} & \Bigl(R(a)m_1R(b)-R_M(R(a)m_1b+am_2S(b)),\\
  & S(a)m_2S(b)-S_M(R(a)m_1b+am_2S(b))\Bigr),
\end{align*}
\begin{align*}
  & a\vartriangleright ((m_1,m_2)\vartriangleleft b) \\
  ={} & a\vartriangleright \Bigl((m_1R(b)-R_M(m_1b),m_2S(b)-S_M(m_1b))\Bigr) \\
  ={} & \Bigl(R(a)(m_1R(b)-R_M(m_1b)) - R_M(am_2S(b)-aS_M(m_1b)),\\
  & S(a)(m_2S(b)-S_M(m_1b))-S_M(am_2S(b)-aS_M(m_1b))\Bigr)\\
  ={} & \Bigl(R(a)m_1R(b)-R_M(R(a)m_1b+am_2S(b)),\\
  & S(a)m_2S(b)-S_M(R(a)m_1b+am_2S(b))\Bigr),
\end{align*}
which give the compatibility of operations $\vartriangleright $ and $\vartriangleleft$:
\[(a\vartriangleright (m_1,m_2))\vartriangleleft b =a\vartriangleright ((m_1,m_2)\vartriangleleft b).\]
\end{proof}

\begin{remark} \rm{(1)} Unlike the Rota-Baxter algebra situation as in \cite[Proposition 4.7]{WZ}, in order to define the new bimodule $\mathscr{D}(M)$ over the associative algebra $A_\star$ we have changed the space from $M$ to $M\oplus M$.

\rm{(2)} If $R$ and $S$ satisfy $R\circ S=S\circ R$, then $(A,\star, R, S)$ is also a Rota-Baxter system (cf. Remark \ref{new-RBS}). In this case,  we can define a Rota-Baxter system bimodule $(\mathscr{D}(M),R_{\mathscr{D}(M)},S_{\mathscr{D}(M)})$ over $(A,\star, R, S)$ by letting $$R_{\mathscr{D}(M)}: M\oplus M\to M\oplus M, (m_1,m_2)\mapsto (R_M(m_1),R_M(m_2)),$$
$$S_{\mathscr{D}(M)}: M\oplus M\to M\oplus M, (m_1,m_2)\mapsto (S_M(m_1),S_M(m_2)).$$
\end{remark}

\section{Cohomology theory of Rota-Baxter systems}
In this section, we will define a cohomology theory for Rota-Baxter systems following (and modifying) the ideas from \cite[Section 5]{WZ}.

\subsection{Cohomology of Rota-Baxter system operators}
Firstly, let's study the cohomology of Rota-Baxter system operators.
Let $(A, \mu, R, S)$ be a Rota-Baxter system and $(M, R_M, S_M)$ be a Rota-Baxter system bimodule over it. Recall
that Proposition \ref{star} and Proposition \ref{triangle} give a new associative algebra $A_\star $ and a new bimodule
$_\vartriangleright \mathscr{D}(M)_\vartriangleleft$ over $A_\star$. Consider the Hochschild cochain complex of $A_\star $  with coefficients in $_\vartriangleright \mathscr{D}(M)_\vartriangleleft$:
\[\mathrm{C}_{\mathrm{Alg}}^\bullet(A_\star,{_\vartriangleright \mathscr{D}(M)_\vartriangleleft}):= \bigoplus_{n=0}^\infty \mathrm{C}_{\mathrm{Alg}}^n(A_\star, {_\vartriangleright \mathscr{D}(M) _\vartriangleleft}).\]
More precisely, for $ n\geqslant 0,$
\[ \mathrm{C}_{\mathrm{Alg}}^n(A_\star, {_\vartriangleright\mathscr{D}(M)_\vartriangleleft})= { \rm Hom}(A^{\otimes n}, M\oplus M)\cong{ \rm Hom}(A^{\otimes n}, M)\oplus { \rm Hom}(A^{\otimes n},M)\]
and its differential
\[\partial^n : \mathrm{C}_{\mathrm{Alg}}^n(A_\star, {_\vartriangleright\mathscr{D}(M)_\vartriangleleft}) \to \mathrm{C}_{\mathrm{Alg}}^{n+1}(A_\star, {_\vartriangleright\mathscr{D}(M)_\vartriangleleft})\]
is defined by the following formula:
\begin{align*}
   & \partial^0(f,g)(a) \\
  ={} & \Bigl(-R(a)f(1_\mathbb{K})+R_M(a\,g(1_\mathbb{K}))+f(1_\mathbb{K})R(a)-R_M(f(1_\mathbb{K})\,a),\\
  & -S(a)g(1_\mathbb{K})+S_M(a\,g(1_\mathbb{K}))+g(1_\mathbb{K})S(a)-S_M(f(1_\mathbb{K})\,a)\Bigr)
\end{align*}
for any $f,\,g\in { \rm Hom}(\mathbb{K},M)$, $a\in A$,
\begin{align*}
  & \partial^n(x,y)(a_{1,n+1}) \\
  ={} & (-1)^{n+1}a_1\vartriangleright (x, y)(a_{2,n+1})\\
  + & \sum_{i=1}^n(-1)^{n-i+1}(x,y)(a_{1,i-1}\otimes a_i\star a_{i+1}\otimes a_{i+2,n+1}) \\
  + & (x, y)(a_{1,n})\vartriangleleft a_{n+1}\\
   ={} & \Bigl((-1)^{n+1} R(a_1)x(a_{2,n+1})-(-1)^{n+1} R_M(a_1 y(a_{2,n+1}))\\
   + & \sum_{i=1}^n(-1)^{n-i+1}x(a_{1,i-1}\otimes R(a_i)a_{i+1}+a_iS(a_{i+1})\otimes a_{i+2,n+1})\\
   + & x(a_{1,n})R(a_{n+1}) - R_M(x(a_{1,n})a_{n+1}), \\
   & (-1)^{n+1}S(a_1)y(a_{2,n+1})- (-1)^{n+1}S_M(a_1 y(a_{2,n+1})) \\
   + & \sum_{i=1}^n(-1)^{n-i+1} y(a_{1,i-1}\otimes R(a_i)a_{i+1}+a_iS(a_{i+1})\otimes a_{i+2,n+1}) \\
   + & y(a_{1,n})S(a_{n+1}) - S_M(x(a_{1,n})a_{n+1}) \Bigr)
 \end{align*}
for $ n\geqslant 1,$ $x,y\in \mathrm{C}_{\mathrm{Alg}}^n(A, M)$ and $a_1,\ldots ,a_{n+1}\in A$.

\begin{defini} Let $A=(A, \mu, R, S)$ be a Rota-Baxter system and $M=(M, R_M, S_M)$ be a Rota-Baxter system bimodule over it. Then the Hochschild cochain complex $\Bigl(\mathrm{C}_{\mathrm{Alg}}^\bullet(A_\star, {_\vartriangleright\mathscr{D}(M)_\vartriangleleft}), \partial^\bullet \Bigr)$ is called the cochain complex of Rota-Baxter system operators $(R,S)$ with coefficients in $(M, R_M, S_M)$, denoted by $\mathrm{C}_{\mathrm{RBSO}}^\bullet(A, M)$.
The cohomology of $\mathrm{C}_{\mathrm{RBSO}}^\bullet(A, M)$, denoted by $\mathrm{H}_{\mathrm{RBSO}}^\bullet(A, M)$, is called the cohomology of Rota-Baxter system operators $(R,S)$ with coefficients in $(M, R_M, S_M)$.
\end{defini}

When $(M, R_M, S_M)$ is the regular Rota-Baxter system bimodule $(A, R, S)$, we denote $\mathrm{C}_{\mathrm{RBSO}}^\bullet(A, A)$ by $\mathrm{C}_{\mathrm{RBSO}}^\bullet(A)$ and call it the cochain complex of Rota-Baxter system operators $(R,S)$, and denoted $\mathrm{H}_{\mathrm{RBSO}}^\bullet(A, A)$ by $\mathrm{H}_{\mathrm{RBSO}}^\bullet(A)$ and call it the cohomology of Rota-Baxter system operators $(R,S)$.

\subsection{Cohomology of Rota-Baxter systems}
In this subsection, we will combine the Hochschild cochain complex of associative algebra $(A,\mu)$ and the cochain complex of Rota-Baxter system operators $(R,S)$ to define a cohomology theory for Rota-Baxter system $(A, \mu, R, S)$.

Let $M=(M, R_M, S_M)$ be a Rota-Baxter system bimodule over a Rota-Baxter system $A=(A, \mu, R, S)$. Let $\mathrm{C}_{\mathrm{Alg}}^\bullet(A,M)$ be the Hochschild cochain complex of $(A,\mu)$ with coefficients in $M$ and $\mathrm{C}_{\mathrm{RBSO}}^\bullet(A,M)$ be the cochain complex of Rota-Baxter system operators $(R,S)$ with coefficients in $(M, R_M, S_M)$. We now define a chain map $\Phi^\bullet : \mathrm{C}_{\mathrm{Alg}}^\bullet(A,M) \to \mathrm{C}_{\mathrm{RBSO}}^\bullet(A,M)$ as follows.

Define $\Phi^0= (\Phi_R^0,\,\Phi_S^0) :\mathrm{C}_{\mathrm{Alg}}^0(A,M)=\mathrm{Hom}(\mathbb{K},M) \to \mathrm{C}_{\mathrm{RBSO}}^0(A,M)\cong \mathrm{Hom}(\mathbb{K},M)\oplus \mathrm{Hom}(\mathbb{K},M)$ by
\[\Phi_R^0 = \Phi_S^0 = \mathrm{Id}_{\mathrm{Hom}(\mathbb{K},M)},\]
 and for $n\geqslant 1$, define\\
  $\Phi^n= (\Phi_R^n,\,\Phi_S^n) :
   \mathrm{C}_{\mathrm{Alg}}^n(A,M)= \mathrm{Hom}(A^{\otimes n},M) \to \mathrm{C}_{\mathrm{RBSO}}^n(A,M)\cong \mathrm{Hom}(A^{\otimes n},M)\oplus \mathrm{Hom}(A^{\otimes n},M)$ by
\begin{align*}
  &\Phi_R^n(f)(a_1\otimes\cdots\otimes a_n)\\
  ={} & f\Bigl(R(a_1)\otimes\cdots\otimes R(a_n)\Bigr) \\
  - & R_M\circ \sum_{i=1}^nf\Bigl(R(a_1)\otimes\cdots\otimes R(a_{i-1})\otimes a_i\otimes S(a_{i+1})\cdots \otimes S(a_n)\Bigr),
\end{align*}
 \begin{align*}
  &\Phi_S^n(f)(a_1\otimes\cdots\otimes a_n)\\
  ={} & f\Bigl(S(a_1)\otimes\cdots\otimes S(a_n)\Bigr) \\
  - & S_M\circ \sum_{i=1}^nf\Bigl(R(a_1)\otimes\cdots\otimes R(a_{i-1})\otimes a_i\otimes S(a_{i+1})\cdots \otimes S(a_n)\Bigr)
\end{align*}
for $f\in \mathrm{C}_{\mathrm{Alg}}^n(A,M)$.

\begin{prop} The map $\Phi^\bullet : \mathrm{C}_{\mathrm{Alg}}^\bullet(A,M) \to \mathrm{C}_{\mathrm{RBSO}}^\bullet(A,M)$ is a chain map, that is, the following diagram commutes:
 \begin{equation*}
 \xymatrix{
  \mathrm{C}_{\mathrm{Alg}}^0(A,M) \ar[r]^{\delta^0}\ar[d]^{\Phi^0} &  \mathrm{C}_{\mathrm{Alg}}^1(A,M) \ar@{-->}[r] \ar[d]^{\Phi^1}
 & \mathrm{C}_{\mathrm{Alg}}^n(A,M) \ar[r]^{\delta^n}\ar[d]^{\Phi^n} &  \mathrm{C}_{\mathrm{Alg}}^{n+1}(A,M) \ar[d]^{\Phi^{n+1}}\cdots\\
\mathrm{C}_{\mathrm{RBSO}}^0(A,M)\ar[r]^{\partial^0} &  \mathrm{C}_{\mathrm{RBSO}}^1(A,M)\ar@{-->}[r] & \mathrm{C}_{\mathrm{RBSO}}^n(A,M) \ar[r]^{\partial^n} &  \mathrm{C}_{\mathrm{RBSO}}^{n+1}(A,M)\cdots
 }
 \end{equation*}
\end{prop}

\begin{proof}
We just need to prove $\partial^n\circ \Phi^n(f)=\Phi^{n+1}\circ \delta^n(f)$ for any $n\geqslant 0$ and for any $f\in \mathrm{C}_{\mathrm{Alg}}^n(A,M)$.

When $n=0, f\in {\rm Hom}(\mathbb{K},M), a\in A$, we have
\begin{align*}
& \partial^0\circ \Phi^0(f)(a) \\
  ={} & \partial^0(f,f)(a)\\
  ={} & \Bigl(-R(a)f(1_{\mathbb{K}})+f(1_{\mathbb{K}})R(a)-R_M(f(1_{\mathbb{K}})a)+R_M(af(1_{\mathbb{K}})),\\
  ={} &-S(a)f(1_{\mathbb{K}})+f(1_{\mathbb{K}})S(a)-S_M(f(1_{\mathbb{K}})a)+S_M(af(1_{\mathbb{K}}))\Bigr);
\end{align*}
on the other hand, we have
\begin{align*}
& \Phi^{1}\circ \delta^0(f)(a) \\
  ={} & \Bigl(\delta^0(f)(R(a))-R_M(\delta^0(f)(a)),\delta^0(f)(S(a))-S_M(\delta^0(f)(a))\Bigr) \\
  ={} & \Bigl(-R(a)f(1_{\mathbb{K}})+f(1_{\mathbb{K}})R(a)+R_M(af(1_{\mathbb{K}})-f(1_{\mathbb{K}})a),\\
  ={} &-S(a)f(1_{\mathbb{K}})+f(1_{\mathbb{K}})S(a)+S_M(af(1_{\mathbb{K}})-f(1_{\mathbb{K}})a)\Bigr).
\end{align*}

This proves $\partial^0\circ \Phi^0=\Phi^1\circ \delta^0$.

When $n\geqslant 1$, for $f\in \mathrm{C}_{\mathrm{Alg}}^n(A, M), a_1,\ldots ,a_{n+1}\in A$, we have (here we write $R\cdot f(a_{1,n+1}):=R(a_1)f(a_{2,n+1})$)

\begin{align*}
   & \partial^n\circ \Phi^n(f) \\
    ={} & \partial^n\circ (\Phi_R^{n},\Phi_S^{n})(f) \\
  ={} & \partial^n \Bigl(f(R^{\otimes n})-R_M\sum_{i=1}^nf(R^{\otimes i-1}\otimes Id\otimes S^{\otimes n-i} ), f(S^{\otimes n})-S_M\sum_{i=1}^nf(R^{\otimes i-1}\otimes Id\otimes S^{\otimes n-i} )\Bigr)\\
  ={} & \Biggl( (-1)^{n+1}R\cdot \Bigl(f(R^{\otimes n})-R_M\sum_{i=1}^nf(R^{\otimes i-1}\otimes Id\otimes S^{\otimes n-i} )\Bigr)\\
   - & (-1)^{n+1}R_M\Bigl\{Id\cdot \Bigl(f(S^{\otimes n})-S_M\sum_{i=1}^nf(R^{\otimes i-1}\otimes Id\otimes S^{\otimes n-i} )\Bigr ) \Bigr\}\\
   + & \sum_{i=1}^{n}(-1)^{n-i+1}f\Bigl(R^{\otimes i-1}\otimes R(R\cdot Id + Id\cdot S)\otimes R^{\otimes n-i}\Bigr )\\
   - & R_M\sum_{i=1}^{n}\Bigl\{\sum_{1\leqslant j\leqslant i-1}(-1)^{n-j+1}f\Bigl(R^{\otimes j-1}\otimes R(R\cdot Id+Id\cdot S)
   \otimes R^{\otimes i-j-2}\otimes Id\otimes S^{\otimes n-i+1} \Bigr)\\
   + & (-1)^{n-i+1}f(R^{\otimes i-1}\otimes (R\cdot Id +Id\cdot S)\otimes S^{\otimes n-i} )\\
   + & \sum_{i+1\leqslant j\leqslant n}(-1)^{n-j+1}f\Bigl(R^{\otimes i-1}\otimes Id\otimes S^{\otimes j-i-1} \otimes
   S( R\cdot Id +Id\cdot S)\otimes S^{\otimes n-j} \Bigr)\Bigr\}\\
   + & \Bigl(f(R^{\otimes n})-R_M\sum_{i=1}^nf(R^{\otimes i-1}\otimes Id\otimes S^{\otimes n-i} )\Bigr)\cdot R\\
   - & R_M \Bigl\{\Bigl(f(R^{\otimes n})-R_M\sum_{i=1}^nf(R^{\otimes i-1}\otimes Id\otimes S^{\otimes n-i} )\Bigr)\cdot Id\Bigr\},
\end{align*}

\begin{align*}
   & (-1)^{n+1}S\cdot \Bigl(f(S^{\otimes n})-S_M\sum_{i=1}^nf(R^{\otimes i-1}\otimes Id\otimes S^{\otimes n-i} )\Bigr)\\
   -& (-1)^{n+1}S_M\Bigl\{Id\cdot \Bigl(f(S^{\otimes n})-S_M\sum_{i=1}^nf(R^{\otimes i-1}\otimes Id\otimes S^{\otimes n-i} )\Bigr ) \Bigr\}\\
   + & \sum_{i=1}^{n}(-1)^{n-i+1}f\Bigl(S^{\otimes i-1}\otimes S(R\cdot Id + Id\cdot S)\otimes S^{\otimes n-i}\Bigr )\\
   - & S_M\sum_{i=1}^{n}\Bigl\{\sum_{1\leqslant j\leqslant i-1}(-1)^{n-j+1}f\Bigl(R^{\otimes j-1}\otimes R(R\cdot Id+Id\cdot S)
   \otimes R^{\otimes i-j-2}\otimes Id\otimes S^{\otimes n-i+1} \Bigr)\\
   + & (-1)^{n-i+1}f(R^{\otimes i-1}\otimes (R\cdot Id +Id\cdot S)\otimes S^{\otimes n-i} )\\
   + & \sum_{i+1\leqslant j\leqslant n}(-1)^{n-j+1}f\Bigl(R^{\otimes i-1}\otimes Id\otimes S^{\otimes j-i-1} \otimes
   S( R\cdot Id +Id\cdot S)\otimes S^{\otimes n-j} \Bigr)\Bigr\}\\
   + & \Bigl(f(S^{\otimes n})-S_M\sum_{i=1}^nf(R^{\otimes i-1}\otimes Id\otimes S^{\otimes n-i} )\Bigr)\cdot S\\
   - & S_M \Bigl\{\Bigl(f(R^{\otimes n})-R_M\sum_{i=1}^nf(R^{\otimes i-1}\otimes Id\otimes S^{\otimes n-i} )\Bigr)\cdot Id\Bigr\}\Biggr),
   \end{align*}
and
\begin{align*}
& \Phi^{n+1}\circ \delta^n(f) \\
   ={}& (\Phi_R^{n+1},\Phi_S^{n+1})\circ \delta^n(f) \\
   ={} & \biggl(\delta^n(f)(R^{\otimes n+1})-R_M\sum_{i=1}^{n+1}\delta^n(f)(R^{\otimes i-1}\otimes Id\otimes S^{\otimes n-i+1} ),\\
    & \delta^n(f)(S^{\otimes n+1})-S_M\sum_{i=1}^{n+1}\delta^n(f)(R^{\otimes i-1}\otimes Id\otimes S^{\otimes n-i+1} )\biggr)\\
   ={} & \Biggl(
   (-1)^{n+1}R\cdot f(R^{\otimes n})+ \sum_{i=1}^{n}(-1)^{n-i+1}f(R^{\otimes i-1}\otimes R\cdot R\otimes R^{\otimes n-i} )+ f(R^{\otimes n})\cdot R\\
   - & R_M\Bigl\{ (-1)^{n+1}Id\cdot f(S^{\otimes n})+(-1)^{n+1}\sum_{i=2}^{n+1}R\cdot f(R^{\otimes i-2}\otimes Id\otimes S^{\otimes n-i+1})\\
     +& \sum_{i=1}^{n+1}\bigl\{\sum_{1\leqslant j\leqslant i-2}(-1)^{n-j+1}f(R^{\otimes j-1}\otimes R\cdot R\otimes R^{\otimes i-j-2}\otimes Id\otimes S^{\otimes n-i+1} )\\
   + & (-1)^{n-i+2}f(R^{\otimes i-2}\otimes R\cdot Id\otimes S^{\otimes n-i+1} )+ (-1)^{n-i+1}f(R^{\otimes i-1}\otimes Id\cdot S\otimes S^{\otimes n-i} )\\
   + & \sum_{i+1\leqslant j\leqslant n}(-1)^{n-j+1} f(R^{\otimes i-1}\otimes Id\otimes S^{\otimes j-i-1} \otimes S\cdot S\otimes S^{\otimes n-j} )\bigr\}\\
   + & \sum_{i=1}^nf(R^{\otimes i-1}\otimes Id\otimes S^{\otimes n-i})\cdot S +f(R^{\otimes n})\cdot Id)\Bigr\},
   \end{align*}
\begin{align*}
  & (-1)^{n+1}S\cdot f(S^{\otimes n})+ \sum_{i=1}^{n}(-1)^{n-i+1}f(S^{\otimes i-1}\otimes S\cdot S\otimes S^{\otimes n-i} )+ f(S^{\otimes n})\cdot S\\
   - &S_M\Bigl\{ (-1)^{n+1}Id\cdot f(S^{\otimes n})+(-1)^{n+1}\sum_{i=2}^{n+1}R\cdot f(R^{\otimes i-2}\otimes Id\otimes S^{\otimes n-i+1})\\
     +& \sum_{i=1}^{n+1}\bigl\{\sum_{1\leqslant j\leqslant i-2}(-1)^{n-j+1}f(R^{\otimes j-1}\otimes R\cdot R\otimes R^{\otimes i-j-2}\otimes Id\otimes S^{\otimes n-i+1} )\\
   + & (-1)^{n-i+2}f(R^{\otimes i-2}\otimes R\cdot Id\otimes S^{\otimes n-i+1} )+ (-1)^{n-i+1}f(R^{\otimes i-1}\otimes Id\cdot S\otimes S^{\otimes n-i} )\\
   + & \sum_{i+1\leqslant j\leqslant n}(-1)^{n-j+1} f(R^{\otimes i-1}\otimes Id\otimes S^{\otimes j-i-1} \otimes S\cdot S\otimes S^{\otimes n-j} )\bigr\}\\
   + & \sum_{i=1}^nf(R^{\otimes i-1}\otimes Id\otimes S^{\otimes n-i})\cdot S +f(R^{\otimes n})\cdot Id)\Bigr\}
   \Biggr).
 \end{align*}
Comparing the above two equations we obtain $\partial^n\circ \Phi^n(f)=\Phi^{n+1}\circ \delta^n(f)$. Hence, $\Phi^\bullet $ is a chain map.
\end{proof}

\begin{defini} Let $M=(M, R_M, S_M)$ be a Rota-Baxter system bimodule over a Rota-Baxter system $A=(A, R, S)$. We define the cochain complex $(\mathrm{C}_{\mathrm{RBS}}^\bullet(A,M), d^\bullet)$ of Rota-Baxter system $(A, R, S)$ with coefficients in $M$ to be the negative shift of the mapping cone of $\Phi^\bullet$, that is, we have
\begin{align*}
\mathrm{C}_{\mathrm{RBS}}^0(A,M) ={} & \mathrm{C}_{\mathrm{Alg}}^0(A,M),\\
\mathrm{C}_{\mathrm{RBS}}^n(A,M) ={} & \mathrm{C}_{\mathrm{Alg}}^n(A,M)\oplus \mathrm{C}_{\mathrm{RBSO}}^{n-1}(A,M), \forall n\geqslant 1,
\end{align*}
and the differential $d^n : \mathrm{C}_{\mathrm{RBS}}^n(A,M)\to \mathrm{C}_{\mathrm{RBS}}^{n+1}(A,M)$ is given by
\begin{equation}
d^n(f, (x, y)) = (\delta^n(f), -\partial^{n-1}(x, y) - \Phi^n(f)) \label{differential}
\end{equation}
for any $f\in \mathrm{C}_{\mathrm{Alg}}^n(A,M), x,y\in \mathrm{C}_{\mathrm{Alg}}^{n-1}(A,M)$. The cohomology of $(\mathrm{C}_{\mathrm{RBS}}^\bullet(A,M), d^\bullet)$, denoted by $\mathrm{H}_{\mathrm{RBS}}^\bullet(A,M)$, is called the cohomology of Rota-Baxter system $(A, R, S)$ with coefficients in $M$. When $(M, R_M, S_M)=(A,  R, S)$, we just denote $\mathrm{C}_{\mathrm{RBS}}^\bullet(A,A), \mathrm{H}_{\mathrm {RBS}}^\bullet(A,A)$ by $\mathrm{C}_{\mathrm{RBS}}^\bullet(A), \mathrm{H}_{\mathrm{RBS}}^\bullet(A)$ respectively, and call them the cochain complex, the cohomology of Rota-Baxter system $(A, R, S)$ respectively.
\end{defini}

There is an obvious short exact sequence of complexes:
\begin{equation}
  \xymatrix@C=0.5cm{
    0 \ar[r] & s\mathrm{C}_{\mathrm{RBSO}}^\bullet(A,M) \ar[r] & \mathrm{C}_{\mathrm{RBS}}^\bullet(A,M) \ar[r] & \mathrm{C}_{\mathrm{Alg}}^\bullet(A,M) \ar[r] & 0 }
\end{equation}
which induces a long exact sequence of cohomology groups
\begin{gather*}
  \xymatrix@C=0.5cm{
    0 \ar[r] & \mathrm{H}_{\mathrm{RBS}}^0(A,M) \ar[r] & {\mathrm{HH}}^0(A,M) \ar[r] & \mathrm{H}_{\mathrm{RBSO}}^0(A,M) \ar[r] & \mathrm{H}_{\mathrm{RBS}}^1(A,M) \ar[r] & \cdots } \\
  \xymatrix@C=0.5cm{
    \cdots \ar[r] & {\mathrm{HH}}^p(A,M) \ar[r] & \mathrm{H}_{\mathrm{RBSO}}^p(A,M) \ar[r] & \mathrm{H}_{\mathrm{RBS}}^{p+1}(A,M) \ar[r] & \cdots}
\end{gather*}

\begin{remark}
Let $(A, R, \lambda)$ be a Rote-Baxter algebra of weight $\lambda$. Then $(A, R, R+\lambda id)$ is a Rote-Baxter system (cf. Remark \ref{R+id}). We observe that there is a monomorphism from the cochain complex $(\mathrm{C}_{\mathrm{RBA}_\lambda}^\bullet(A), d^\bullet)$ of Rote-Baxter algebra $(A, R, \lambda)$ defined in \cite{WZ} to the cochain complex $(\mathrm{C}_{\mathrm{RBS}}^\bullet(A), d^\bullet)$ of Rote-Baxter system $(A, R, R+\lambda id)$:
\[ \psi^n = \left(
\begin{array}{cc}
1 & 0\\
0 & 1\\
0 & 1\\
\end{array}\right) : \mathrm{C}_{\mathrm{RBA}_\lambda}^n(A) = \mathrm{C}_{\mathrm{Alg}}^{n}(A)\oplus \mathrm{C}_{\mathrm{Alg}}^{n-1}(A) \to \mathrm{C}_{\mathrm{RBS}}^n(A) \cong \mathrm{C}_{\mathrm{Alg}}^{n}(A)\oplus \mathrm{C}_{\mathrm{Alg}}^{n-1}(A)\oplus \mathrm{C}_{\mathrm{Alg}}^{n-1}(A)\]
for all $n\geqslant 0$ and with $\mathrm{Coker}\psi^n = \mathrm{C}_{\mathrm{Alg}}^{n-1}(A)=\mathrm{Hom}(A^{\otimes n-1},A)$. Furthermore, there is a short exact sequence of complexes:
\begin{equation*}
  \xymatrix@C=0.5cm{
    0 \ar[r] & \mathrm{C}_{\mathrm{RBA}_\lambda}^\bullet(A) \ar[r]^{\psi} & \mathrm{C}_{\mathrm{RBS}}^\bullet(A) \ar[r] & \mathrm{Coker}\psi \ar[r] & 0, }
\end{equation*}
where the differential $\bar{d}^n:\mathrm{Coker}\psi^n=\mathrm{Hom}(A^{\otimes n-1},A) \to \mathrm{Coker}\psi^{n+1}=\mathrm{Hom}(A^{\otimes n},A)$ of the complex $\mathrm{Coker}\psi$ is given by
\begin{align*}
   & \bar{d}^n(h)(a_{1,n}) \\
  ={} & (-1)^{n-1} R(a_1)h(a_{2,n}) + \sum_{i=1}^{n-1} (-1)^{n-i-1} h(a_{1,i-1}\otimes (R(a_i)a_{i+1} + a_i(R+\lambda )(a_{i+1}))\otimes a_{i+2,n} ) \\
  & - h(a_{1,n-1})(R+\lambda )(a_{n}).
\end{align*}
\end{remark}

\section{Formal deformations of Rota-Baxter systems and cohomological interpretation}
In this section, we will study formal deformations of Rota-Baxter systems and interpret them
via lower degree cohomology groups of Rota-Baxter systems defined in last section.

\subsection{Formal deformations of Rota-Baxter systems}
Let $(A, \mu, R, S)$ be a Rota-Baxter system. Consider a 1-parameter family:
\begin{gather*}
  \mu_t = \sum_{i=0}^\infty \mu_it^i, \quad \mu_i\in \mathrm{C}_{\mathrm{Alg}}^2(A),  \\
  R_t = \sum_{i=0}^\infty R_it^i, \quad S_t = \sum_{i=0}^\infty S_it^i,\quad R_i,S_i\in \mathrm{C}_{\mathrm{Alg}}^1(A).
\end{gather*}

\begin{defini} A 1-parameter formal deformation of Rota-Baxter system $(A, \mu, R, S)$  is a triple $(\mu_t, R_t, S_t)$ which endows the  $\Bbbk[[t]]$-module $A[[t]]$ with a Rota-Baxter system structure over $\Bbbk[[t]]$ such that $(\mu_0, R_0, S_0)$ = $(\mu, R, S)$.
\end{defini}

Power series $\mu_t$, $R_t$ and $S_t$ determine a 1-parameter formal deformation of Rota-Baxter system
$(A, \mu, R, S)$ if and only if for any $a, b, c\in A$, the following equations hold:
\begin{gather*}
  \mu_t(a\otimes \mu_t(b\otimes c)) = \mu_t(\mu_t(a\otimes b)\otimes c), \\
  \mu_t(R_t(a)\otimes R_t(b)) = R_t\Bigl(\mu_t(R_t(a)\otimes b) + \mu_t(a\otimes S_t(b))\Bigr), \\
  \mu_t(S_t(a)\otimes S_t(b)) = S_t\Bigl(\mu_t(R_t(a)\otimes b) + \mu_t(a\otimes S_t(b))\Bigr).
\end{gather*}

By expanding these equations and comparing the coefficients of $t^n$, we obtain that $\{\mu_i\}_{i\geq0}$, $\{R_i\}_{i\geq0}$ and $\{S_i\}_{i\geq0}$ have to satisfy: for any $n\geqslant 0$,
\begin{gather}
  \sum_{i=0}^n \mu_i\circ (\mu_{n-i}\otimes Id) = \sum_{i=0}^n \mu_i\circ (Id\otimes \mu_{n-i}),\label{mu1} \\
  \sum_{\substack{i+j+k=n \\
  i,j,k\geqslant0}}
  \mu_i\circ (R_j\otimes R_k) = \sum_{\substack{i+j+k=n \\
  i,j,k\geqslant0}}
  R_i\circ \mu_j\circ (R_k\otimes Id) + \sum_{\substack{i+j+k=n \\
  i,j,k\geqslant0}}
  R_i\circ \mu_j\circ (Id\otimes S_k),\label{R1}\\
  \sum_{\substack{i+j+k=n \\
  i,j,k\geqslant0}}
  \mu_i\circ (S_j\otimes S_k) = \sum_{\substack{i+j+k=n \\
  i,j,k\geqslant0}}
  S_i\circ \mu_j\circ (R_k\otimes Id) + \sum_{\substack{i+j+k=n \\
  i,j,k\geqslant0}}
  S_i\circ \mu_j\circ (Id\otimes S_k).\label{S1}
  \end{gather}

Obviously, when $n = 0$, the above conditions are exactly the associativity of $\mu = \mu_0$ and Equations (\ref{3})(\ref{4}) which are the defining relations of Rota-Baxter system operators $R_0 = R$, $S_0 = S$ respectively.

Recall that we have constructed a cochain complex $\mathrm{C}_{\mathrm{RBS}}^\bullet(A)$ of Rota-Baxter system $(A, R, S)$ in last section.

\begin{prop} \rm{(Compare to \cite[Proposition 6.2]{WZ})} \label{2-co}Let $(A[[t]], \mu_t, R_t, S_t)$ be a 1-parameter formal deformation of Rota-Baxter system
$(A, \mu, R, S)$. Then $(\mu_1, R_1, S_1)$ is a 2-cocycle in the cochain complex $\mathrm{C}_{\mathrm{RBS}}^\bullet(A)$ of Rota-Baxter system $(A, R, S)$.
\end{prop}

\begin{proof}
When $n=1$, Equation (\ref{mu1}) becomes
\[\mu_1\circ(\mu\otimes Id)+\mu\circ(\mu_1\otimes Id)=\mu_1\circ (Id\otimes \mu)+\mu\circ (Id\otimes \mu_1),\]
and Equations (\ref{R1})(\ref{S1}) become
\begin{align*}
 & \mu_1\circ(R\otimes R)-\{R\circ\mu_1\circ(R\otimes Id)+R\circ\mu_1\circ(Id\otimes S)\} \\
 ={} & -\{\mu\circ(R\otimes R_1)-R\circ\mu\circ(Id\otimes S_1)\} + \{R_1\circ\mu\circ(R\otimes Id)+R_1\circ\mu\circ(Id\otimes S)\}\\
 & - \{\mu\circ(R_1\otimes R)-R\circ\mu\circ(R_1\otimes Id)\},
 \end{align*}
 \begin{align*}
  & \mu_1\circ(S\otimes S)-\{S\circ\mu_1\circ(R\otimes Id)+S\circ\mu_1\circ(Id\otimes S)\} \\
 ={} & -\{\mu\circ(S\otimes S_1)-S\circ\mu\circ(Id\otimes S_1)\} + \{S_1\circ\mu\circ(R\otimes Id)+S_1\circ\mu\circ(Id\otimes S)\}\\
 & - \{\mu\circ(S_1\otimes S)-S\circ\mu\circ(R_1\otimes Id)\}.
\end{align*}
Note that the first equation is exactly $\delta^2(\mu_1)=0\in \mathrm{C}_{\mathrm{Alg}}^\bullet(A)$ and the second and the third equations show
\[\Phi^2(\mu_1) = -\partial^1(R_1,S_1)\in \mathrm{C}_{\mathrm{RBSO}}^\bullet(A).\]
So $(\mu_1, R_1, S_1)$ is a 2-cocycle in $\mathrm{C}_{\mathrm{RBS}}^\bullet(A)$ by the formula (\ref{differential}) of the differential $d$.
\end{proof}

\begin{defini} The 2-cocycle $(\mu_1, R_1, S_1)$ is called the infinitesimal of the 1-parameter formal deformation $(A[[t]], \mu_t, R_t, S_t)$ of Rota-Baxter system $(A, \mu, R, S)$.
\end{defini}

\begin{defini}
Let $(A[[t]], \mu_t, R_t, S_t)$ and $(A[[t]], \mu_t^\prime, R_t^\prime, S_t^\prime)$ be two 1-parameter formal deformations of Rota-Baxter system $(A, \mu, R, S)$. A formal isomorphism from the formal deformation $(A[[t]], \mu_t^\prime, R_t^\prime, S_t^\prime)$ to $(A[[t]], \mu_t, R_t, S_t)$ is a power series $\Psi _t = \sum_{i=0} \Psi_it^i : A[[t]] \to A[[t]]$, where $\Psi_i : A \to A$ are linear maps with $\Psi_0 = Id_A$, such that:
\begin{gather}
 \Psi_t\circ \mu_t^\prime = \mu_t\circ (\Psi_t\otimes \Psi_t),\label{mut}\\
 \Psi_t\circ R_t^\prime= R_t\circ \Psi_t,\label{rt}\\
 \Psi_t\circ S_t^\prime = S_t\circ \Psi_t.\label{st}
 \end{gather}
 In this case, we say that the two formal deformations $(A[[t]], \mu_t, R_t, S_t)$ and $(A[[t]], \mu_t^\prime, R_t^\prime, S_t^\prime)$ are equivalent.
 \end{defini}

Given a Rota-Baxter system $(A, \mu, R, S)$, the power series $\mu_t, R_t, S_t$ with $\mu_i = \delta_{i,0} \mu,\, R_i = \delta_{i,0} R,\, S_i = \delta_{i,0} S$ make $(A[[t]], \mu_t, R_t, S_t)$ into a $1$-parameter formal deformation of $(A, \mu, R, S)$. Formal deformations equivalent to this one are called trivial.

\begin{thm} The infinitesimals of two equivalent 1-parameter formal deformations of $(A, \mu, R, S)$
are in the same cohomology class in $\mathrm{H}_{\mathrm{RBS}}^2(A)$.
\end{thm}

\begin{proof} Let $\Psi _t : (A[[t]], \mu_t^\prime, R_t^\prime, S_t^\prime) \to (A[[t]], \mu_t, R_t, S_t)$ be a formal isomorphism.  Expanding the identities and collecting coefficients of $t$, we get from Equations (\ref{mut})-(\ref{st})
\begin{gather*}
 \mu_1^\prime = \mu_1 + \mu \circ (Id\otimes \Psi_1) - \Psi_1\circ \mu + \mu \circ (\Psi_1\otimes Id),\\
 R_1^\prime = R_1 + R\circ \Psi_1 - \Psi_1\circ R,\\
 S_1^\prime = S_1 + S\circ \Psi_1 - \Psi_1\circ S.
\end{gather*}
That is, we have
\[(\mu_1^\prime, R_1^\prime, S_1^\prime)-(\mu_1,R_1,S_1)=(\delta^1(\Psi_1),-\Phi^1(\Psi_1))=d^1(\Psi_1,0,0)\in \mathrm{C}_{\mathrm{RBS}}^2(A).\]
\end{proof}

\begin{defini} A Rota-Baxter system $(A, \mu, R, S)$ is said to be rigid if every 1-parameter formal deformation is trivial.
\end{defini}

\begin{thm} Let $(A, \mu, R, S)$ be a Rota-Baxter system. If $\mathrm{H}_{\mathrm{RBS}}^2(A) = 0$, then $(A, \mu, R, S)$ is rigid.
\end{thm}

\begin{proof}
Let $(A[[t]], \mu_t, R_t, S_t)$ be a 1-parameter formal deformation of $(A, \mu, R, S)$. By Proposition \ref{2-co}, $(\mu_1, R_1, S_1)$ is a 2-cocycle. By $\mathrm{H}_{\mathrm{RBS}}^2(A) = 0$, there exists a 1-cochain
\[(\Psi_1^\prime,f,f)\in \mathrm{C}_{\mathrm{RBS}}^1(A)=\mathrm{C}_{\mathrm{Alg}}^1(A)\oplus({ \rm Hom}(\mathbb{K},A)\oplus { \rm Hom}(\mathbb{K},A)) \]
such that $(\mu_1, R_1, S_1)=d^1(\Psi_1^\prime,(f,f))$, that is, $\mu_1=\delta^1(\Psi_1^\prime)$ and $(R_1, S_1)= -\partial^0(f,f)-\Phi^1(\Psi_1^\prime)$. Let $\Psi_1=\Psi_1^\prime+\delta^0(f)$. Since $\Phi^1(\delta^0(f))=\partial^0(\Phi^0(f))=\partial^0(f,f)$, then
\[\mu_1=\delta^1(\Psi_1)=\mu\circ(Id\otimes\Psi_1)-\Psi_1\circ\mu+\mu\circ(\Psi_1\otimes Id),\]
and
\[ (R_1, S_1)=-\Phi^1(\Psi_1)=(-\Psi_1\circ R+R\circ\Psi_1,-\Psi_1\circ S+S\circ\Psi_1).\]
Setting $\Psi_t^{(1)}=Id_A-\Psi_1t:A[[t]]\to A[[t]]$, then $(\Psi_t^{(1)})^{-1}=Id_A+\Psi_1t+\Psi_2t^2+\cdots $, and we have a deformation $(A[[t]], \mu_t^{(1)}, R_t^{(1)}, S_t^{(1)})$, where
\begin{gather*}
\mu_t^{(1)} = (\Psi_t^{(1)})^{-1}\circ \mu_t\circ (\Psi_t^{(1)}\otimes \Psi_t^{(1)}),\\
  R_t^{(1)}= (\Psi_t^{(1)})^{-1}\circ R_t\circ \Psi_t^{(1)},\\
 S_t^{(1)} = (\Psi_t^{(1)})^{-1}\circ S_t\circ \Psi_t^{(1)}.
\end{gather*}
So we have
\begin{gather*}
\mu_1^{(1)}=\Psi_1\circ\mu+\mu_1+\mu(-\Psi_1\otimes Id)+\mu(Id\otimes(-\Psi_1))=0,\\
 R_1^{(1)} = \Psi_1\circ R+R_1+R\circ (-\Psi_1)=0,\\
 S_1^{(1)} = \Psi_1\circ S+S_1+S\circ (-\Psi_1)=0.
\end{gather*}
Then
\begin{gather*}
  \mu_t^{(1)} = \mu+\mu_2^{(1)}t^2+\cdots,\\
  R_t^{(1)}=R+R_2^{(1)}t^2+\cdots,\\
   S_t^{(1)} =S+ S_2^{(1)}t^2+\cdots.
\end{gather*}
Assume there exists $\Psi_t^{(1)}=Id_A-\Psi_1t,\cdots,\Psi_t^{(n)}=Id_A-\Psi_nt^n$ such that for
\begin{gather*}
\mu_t^{(n)} = (\Psi_t^{(n)})^{-1}\circ\cdots \circ(\Psi_t^{(1)})^{-1}\circ \mu_t\circ \Bigl((\Psi_t^{(1)}\circ\cdots \circ\Psi_t^{(n)})\otimes (\Psi_t^{(1)}\circ\cdots \circ\Psi_t^{(n)})\Bigr),\\
  R_t^{(n)}= (\Psi_t^{(n)})^{-1}\circ\cdots \circ(\Psi_t^{(1)})^{-1}\circ R_t\circ (\Psi_t^{(1)}\circ\cdots \circ\Psi_t^{(n)}),\\
 S_t^{(n)} = (\Psi_t^{(n)})^{-1}\circ\cdots \circ(\Psi_t^{(1)})^{-1}\circ S_t\circ (\Psi_t^{(1)}\circ\cdots \circ\Psi_t^{(n)}),
\end{gather*}
we have
 \begin{gather*}
   \mu_1^{(n)}=\cdots =\mu_n^{(n)} =0,\\
   R_1^{(n)}=\cdots =R_n^{(n)} =0,\\
   S_1^{(n)}=\cdots =S_n^{(n)} =0.
 \end{gather*}
By Equations (\ref{R1}) and (\ref{S1}) , $\Bigl(\mu_{n+1}^{(n)},(R_{n+1}^{(n)},S_{n+1}^{(n)})\Bigr)$ is a 2-cocycle in $\mathrm{C}_{\mathrm{RBS}}^\bullet(A)$. Then there exists $\Psi_{n+1}\in { \rm Hom}(A,A)$ such that $d^1\Bigl(\Psi_{n+1},(0,0)\Bigr)=\Bigl(\mu_{n+1}^{(n)},(R_{n+1}^{(n)},S_{n+1}^{(n)})\Bigr)$. Let $\Psi_t^{(n+1)}=Id_A-\Psi_{n+1}t^{n+1}$, and
 \begin{gather*}
\mu_t^{(n+1)} = (\Psi_t^{(n+1)})^{-1}\circ \mu_t^{(n)}\circ (\Psi_t^{(n+1)}\otimes \Psi_t^{(n+1)}),\\
  R_t^{(n+1)}= (\Psi_t^{(n+1)})^{-1}\circ R_t^{(n)}\circ \Psi_t^{(n+1)},\\
 S_t^{(n+1)} = (\Psi_t^{(n+1)})^{-1}\circ S_t^{(n)}\circ \Psi_t^{(n+1)},
\end{gather*}
then
\begin{gather*}
   \mu_1^{(n+1)}=\cdots =\mu_{n+1}^{(n+1)} =0,\\
   R_1^{(n+1)}=\cdots =R_{n+1}^{(n+1)} =0,\\
   S_1^{(n+1)}=\cdots =S_{n+1}^{(n+1)} =0.
 \end{gather*}
Define $\Psi_t=\Psi_t^{(1)}\circ\Psi_t^{(2)}\circ\cdots \circ\Psi_t^{(n)}\circ\cdots:A[[t]]\to A[[t]],$ and
 \begin{gather*}
\mu_t^\prime = \Psi_t^{-1}\circ \mu_t\circ (\Psi_t\otimes \Psi_t),\\
  R_t^\prime= \Psi_t^{-1}\circ R_t\circ \Psi_t,\\
 S_t^\prime = \Psi_t^{-1}\circ S_t\circ \Psi_t.
\end{gather*}
Since
\begin{gather*}
\Psi_t=\Psi_t^{(1)}\circ\Psi_t^{(2)}\circ\cdots \circ\Psi_t^{(n)} \text{mod} (t^{n+1}),\\
(\Psi_t)^{-1}=(\Psi_t^{(n)})^{-1}\circ\cdots \circ(\Psi_t^{(1)})^{-1}  \text{mod}(t^{n+1}),\\
\mu_t^\prime = \mu_t^{(n)} \text{mod} (t^{n+1}),\\
  R_t^\prime= R_t^{(n)}\text{mod} (t^{n+1}),\\
 S_t^\prime = S_t^{(n)} \text{mod}(t^{n+1}),
\end{gather*}
we have
\begin{gather*}
   \mu_1^\prime=\cdots =\mu_{n+1}^\prime =0,\\
   R_1^\prime=\cdots =R_{n+1}^\prime =0,\\
   S_1^\prime=\cdots =S_{n+1}^\prime =0,
 \end{gather*}
for all $n\geqslant 1$. \\
Hence $\mu_t^\prime = \mu, R_t^\prime= R$ and $S_t^\prime = S$. So $(A[[t]], \mu_t, R_t, S_t)$ is equivalant to $(A[[t]], \mu, R, S)$, that is,  $(A, \mu, R, S)$ is rigid.
\end{proof}

\subsection{Formal deformations of Rota-Baxter system operators with multiplication fixed}
Let $(A, \mu = \cdot ,R,S)$ be a Rota-Baxter system. Let us consider the case where we only
deform the Rota-Baxter system operators with the multiplication fixed. So $A[[t]]=\{\sum _{i=0}^\infty a_it^i | a_i\in A, \forall i\geqslant 0\}$ is endowed with
the multiplication induced from that of $A$, say,
\[
(\sum _{i=0}^\infty a_i t^i )(\sum _{j=0}^\infty b_j t^j )=\sum _{n=0}^\infty (\sum_{\substack {i+j=n \\
i,j\geqslant 0}}
 a_i b_j)t^n.
\]
Then $A[[t]]$ becomes a free (???) $\mathbb{K}[[t]]$-algebra, whose  multiplication is still denoted by $\mu$.
In this case, a 1-parameter formal deformation $(\mu_t,R_t,S_t)$ of Rota-Baxter system $(A, \mu ,R,S)$ satisfies $\mu_i=0,\forall i\geqslant 1 $.
So Equation (\ref{mu1}) degenerates and Equations (\ref{R1})(\ref{S1}) become
\begin{gather*}
  \mu\circ(R_t\otimes R_t) = R_t\circ \mu(R_t\otimes Id + Id\otimes S_t), \\
  \mu\circ(S_t\otimes S_t) = S_t\circ \mu(R_t\otimes Id + Id\otimes S_t).
\end{gather*}
Expanding these equations and comparing the coefficients of $t^n$, we obtain that $\{R_i\}_{i\geqslant 0}$, $\{S_i\}_{i\geqslant 0}$ have to
satisfy: for any $n\geqslant 0$,
\begin{gather}
  \sum_{\substack{i+j=n \\
  i,j\geqslant0}}
  \mu\circ (R_i\otimes R_j) = \sum_{\substack{i+j=n \\
  i,j\geqslant0}}
  R_i\circ \mu\circ (R_j\otimes Id) + \sum_{\substack{i+j=n \\
  i,j\geqslant0}}
  R_i\circ \mu\circ (Id\otimes S_j),\label{r0}\\
  \sum_{\substack{i+j=n \\
  i,j\geqslant0}}
  \mu\circ (S_i\otimes S_j) = \sum_{\substack{i+j=n \\
  i,j\geqslant0}}
  S_i\circ \mu\circ (R_j\otimes Id) + \sum_{\substack{i+j=n \\
  i,j\geqslant0}}
  S_i\circ \mu\circ (Id\otimes S_j).\label{s0}
\end{gather}

Obviously, when $n=0$, Equations (\ref{r0})(\ref{s0}) become exactly Equations (\ref{3})(\ref{4}) defining Rota-Baxter system operators $(R=R_0, S=S_0)$.\\
When $n = 1$, Equations (\ref{r0})(\ref{s0}) has the form
\begin{align*}
 & \mu\circ (R\otimes R_1 +R_1\otimes R) \\
 ={} & R\circ \mu\circ (R_1\otimes Id)+ R_1\circ \mu\circ (R\otimes Id+Id\otimes S)+ R\circ \mu\circ (Id\otimes S_1),\\
 & \mu\circ (S\otimes S_1 +S_1\otimes S) \\
 ={} & S\circ \mu\circ (R_1\otimes Id)+ S_1\circ \mu\circ (R\otimes Id+Id\otimes S)+ S\circ \mu\circ (Id\otimes S_1),
\end{align*}
which say exactly that $\partial^1(R_1,S_1)=(0,0)\in \mathrm{C}_{\mathrm{RBSO}}^\bullet(A)$, where $\mathrm{C}_{\mathrm{RBSO}}^\bullet(A)$ is the cochain complex of Rota-Baxter system operators $(R,S)$ defined in last section. This proves the following result:

\begin{prop} Let $R_t,S_t$ be a 1-parameter formal deformation of Rota-Baxter system operators $(R,S)$. Then  $(R_1,S_1)$ is a 1-cocycle in the cochain complex
$\mathrm{C}_{\mathrm{RBSO}}^\bullet(A)$.
\end{prop}

This means that the cochain complex $\mathrm{C}_{\mathrm{RBSO}}^\bullet(A)$ controls formal deformations of Rota-Baxter system operators.

\section{Abelian extensions of Rota-Baxter systems}
In this section, we study abelian extensions of Rota-Baxter systems and show that they are
classified by the second cohomology, as one would expect of a good cohomology theory.
Notice that a vector space $M$ together with two linear transformations $R_M,S_M : M\to M$ is naturally a
Rota-Baxter system where the multiplication on $M$ is defined to be $uv = 0$ for all $u,v\in M$.

\begin{defini}
An abelian extension of Rota-Baxter systems is a short exact sequence of morphisms of Rota-Baxter systems
\begin{equation}\label{ses}
   \xymatrix@C=0.5cm{
    0 \ar[r] & (M,R_M,S_M)\ar[r]^i & (\hat{A},\hat{R},\hat{S}) \ar[r]^p & (A,R,S) \ar[r] & 0 },
\end{equation}
that is, there exist two commutative diagrams:
$$\xymatrix@C=0.5cm{
    0 \ar[r] & M  \ar[r]^i \ar[d]_{R_M} & \hat{A} \ar[r]^p \ar[d]^{\hat{R}} & A \ar[r] \ar[d]^R & 0 \quad &0 \ar[r] & M  \ar[r]^i \ar[d]_{S_M} & \hat{A} \ar[r]^p \ar[d]^{\hat{S}} & A \ar[r] \ar[d]^S & 0 \\
    0 \ar[r] & M \ar[r]^i  & \hat{A} \ar[r]^p  & A \ar[r]  & 0, \quad & 0 \ar[r] & M \ar[r]^i  & \hat{A} \ar[r]^p  & A \ar[r]  & 0,}$$
where the Rota-Baxter system $(M,R_M,S_M)$ satisfies $uv = 0$ for all $u,v\in M$.
\end{defini}

We will call $(\hat{A},\hat{R},\hat{S})$ an abelian extension of $(A,R,S)$ by $(M,R_M,S_M)$.

\begin{defini}
Let $(\hat{A_1},\hat{R_1},\hat{S_1})$ and $(\hat{A_2},\hat{R_2},\hat{S_2})$ be two abelian extensions of $(A,R,S)$ by $(M,R_M,S_M)$. They are said to be isomorphic if there exists an isomorphism of Rota-Baxter systems $\zeta: (\hat{A_1},\hat{R_1},\hat{S_1})\to (\hat{A_2},\hat{R_2},\hat{S_2}) $ such that the following commutative diagram holds:
\begin{equation}
\begin{gathered}\label{iso}
\xymatrix@C=0.5cm{
    0 \ar[r] & (M,R_M,S_M)\ar[r]^{i_1} \ar@{=}[d] & (\hat{A_1},\hat{R_1},\hat{S_1}) \ar[r]^{p_1} \ar[d]_{\zeta} & (A,R,S) \ar[r] \ar@{=}[d] & 0 \\
    0 \ar[r] & (M,R_M,S_M)\ar[r]^{i_2} & (\hat{A_2},\hat{R_2},\hat{S_2}) \ar[r]^{p_2} & (A,R,S) \ar[r] & 0.
    }
\end{gathered}
\end{equation}
\end{defini}

A section of an abelian extension $(\hat{A},\hat{R},\hat{S})$ of $(A,R,S)$ by $(M,R_M,S_M)$ is a linear map $t: A\to \hat{A}$ such that $p\circ t=\mathrm{Id}_A$.\\
We will show that isomorphism classes of abelian extensions of $(A,R,S)$ by $(M,R_M,S_M)$ are in bijection with the second cohomology group $\mathrm{H}_{\mathrm{RBS}}^2(A,M)$.\\
Let $(\hat{A},\hat{R},\hat{S})$ be an abelian extension of $(A,R,S)$ by $(M,R_M,S_M)$ having the form Equation (\ref{ses}). Choose a section $t: A\to \hat{A}$. We define
\[am:=t(a)m,\quad ma:=mt(a), \quad \forall a\in A,m\in M.\]

\begin{prop}\label{bimod}
With the above notations, $(M,R_M,S_M)$ is a Rota-Baxter system bimodule over  $(A,R,S)$.
\end{prop}

\begin{proof}
For any $a,b\in A, m\in M$, since $t(ab)-t(a)t(b)\in M$ implies $t(ab)m=t(a)t(b)m$, we have
\[(ab)m=t(ab)m=t(a)t(b)m=a(bm).\]
Hence, this gives a left $A$-module structure and the case of right module structure is similar.

Moreover, $\hat{R}(t(a))-t(R(a))\in M$ means that $\hat{R}(t(a))m=t(R(a))m$. Thus we have
\begin{align*}
  R(a)R_M(m)={} & t(R(a))R_M(m) \\
  ={} & \hat{R}(t(a))R_M(m) \\
  ={} & R_M(\hat{R}(t(a))m+t(a)S_M(m))\\
  ={} & R_M(R(a)m+aS_M(m)).
\end{align*}
It is similar to see $$R_M(m)R(a)=R_M(R_M(m)a+mS(a)),$$ $$S(a)S_M(m)=S_M(R(a)m+aS_M(m)),$$ $$S_M(m)S(a)=S_M(R_M(m)a+mS(a)).$$
Hence, $(M,R_M,S_M)$ is a Rota-Baxter system bimodule over  $(A,R,S)$.
\end{proof}

We further define linear maps $\Psi : A \otimes A \to M$ and $\chi_R, \chi_S  : A \to M$  respectively by
\begin{align*}
  \Psi(a\otimes b)={} & t(a)t(b)-t(ab), \quad \forall a,b\in A, \\
  \chi_R(a)={} & \hat{R}(t(a))-t(R(a)),\quad \forall a\in A,\\
  \chi_S(a)={} & \hat{S}(t(a))-t(S(a)), \quad \forall a\in A.
\end{align*}

\begin{prop}
The triple $(\Psi, \chi_R, \chi_S)$ is a 2-cocycle of Rota-Baxter system $(A, R, S)$ with coefficients in
the Rota-Baxter system bimodule $(M, R_M, S_M)$ introduced in Proposition \ref{bimod}.
\end{prop}

\begin{proof}
Since $d(\Psi, \chi_R, \chi_S)=(\delta^2(\Psi),-\Phi^2(\Psi)-\partial^1(\chi_R, \chi_S))$, we just need to prove $$\delta^2(\Psi)=0,\quad \Phi^2(\Psi)+\partial^1(\chi_R, \chi_S)=0.$$

On the other hand, since $t\circ R-\hat{R}\circ t,\, t\circ S-\hat{S}\circ t\in M$, we have
\begin{align*}
  & t(R(a))\hat{R}(t(b))-R_M(t(a)\hat{S}(t(b))+\hat{R}(t(a))t(b))+(\hat{R}(t(a))-t(R(a)))t(R(b))\\
  ={} & (\hat{R}(t(a))-t(R(a)))(t(R(b))-\hat{R}(t(b)))=0,\\
  & t(S(a))\hat{S}(t(b))-S_M(t(a)\hat{S}(t(b))+\hat{R}(t(a))t(b))+(\hat{S}(t(a))-t(S(a)))t(S(b))\\
  ={} & (\hat{S}(t(a))-t(S(a)))(t(S(b))-\hat{S}(t(b)))=0.
\end{align*}
Furthermore, we have
\begin{align*}
   & \delta^2(\Psi)(a\otimes b\otimes c)\\
  ={} & -a\Psi(b\otimes c)+\Psi(ab\otimes c)-\Psi(a\otimes bc)+\Psi(a\otimes b)c \\
  ={} & -t(a)\bigl(t(b)t(c)-t(bc)\bigr)+\bigl(t(ab)t(c)-t(abc)\bigr)\\
  - & \bigl(t(a)t(bc)-t(abc)\bigr)+\bigl(t(a)t(b)-t(ab)\bigr)t(c)\\
  ={} & 0,
\\&
\\
  & \Phi^2(\Psi)(a\otimes b) \\
  ={} & \Bigl(\Psi(R(a)\otimes R(b))-R_M(\Psi(R(a)\otimes b+a\otimes S(b))),\\
  & \Psi(S(a)\otimes S(b))-S_M(\Psi(R(a)\otimes b+a\otimes S(b)))\Bigr)\\
   ={} & \Bigl(t(R(a))t(R(b))-t(R(a)R(b))-R_M(t(R(a))t(b)+t(a)t(S(b))-t(R(a)b+aS(b))),\\
  & t(S(a))t(S(b))-t(S(a)S(b))-S_M(t(R(a))t(b)+t(a)t(S(b))-t(R(a)b+aS(b)))\Bigr),
\\&
\\
  & \partial^1(\chi_R, \chi_S)(a\otimes b)\\
  ={} & \Bigl(R(a)\chi_R(b)-R_M(a\chi_S(b))-\chi_R(R(a)b+aS(b))+\chi_R(a)R(b)-R_M(\chi_R(a)b), \\
   & S(a)\chi_S(b)-S_M(a\chi_S(b))-\chi_S(R(a)b+aS(b))+\chi_S(a)S(b)-S_M(\chi_R(a)b)\Bigr)\\
   ={} & \Bigl(R(a)( \hat{R}(t(b))-t(R(b)))-R_M(a(\hat{S}(t(b))-t(S(b))))-\hat{R}(t(R(a)b+aS(b)))\\
   - & t(R(a)R(b))+(\hat{R}(t(a))-t(R(a)))R(b)-R_M((\hat{R}(t(a))-t(R(a)))b), \\
   & S(a)( \hat{S}(t(b))-t(S(b)))-S_M(a(\hat{S}(t(b))-t(S(b))))-\hat{S}(t(R(a)b+aS(b)))\\
   - & t(S(a)S(b))+(\hat{S}(t(a))-t(S(a)))S(b)-S_M((\hat{R}(t(a))-t(R(a)))b)\Bigr)\\
   ={} & -\Bigl(t(R(a))t(R(b))-t(R(a)R(b))-R_M(t(R(a))t(b)+t(a)t(S(b))-t(R(a)b+aS(b))),\\
  & t(S(a))t(S(b))-t(S(a)S(b))-S_M(t(R(a))t(b)+t(a)t(S(b))-t(R(a)b+aS(b)))\Bigr).
\end{align*}
This finishes the proof of our proposition.
\end{proof}

The choice of the section $t$ in fact determines a splitting
$$\xymatrix@C=0.5cm{
    0 \ar[r] & M \ar@<1ex>[r]^i & \hat{A} \ar@<1ex>[r]^p \ar@<1ex>[l]^s & A \ar[r] \ar@<1ex>[l]^t & 0
    }$$
subject to $s\circ i = \mathrm{Id}_M, s\circ t = 0$ and $is+tp=\mathrm{Id}_{\hat{A}}$. Then there is an induced isomorphism of vector spaces
\[\binom{\,p\, }{ s }:\hat{A}\cong A\oplus M :\Bigl(\; t \quad i\;\Bigr) .\]
We can transfer the Rota-Baxter system structure on $\hat{A}$ to $A\oplus M$ via this isomorphism. It is direct to verify that this endows $A\oplus M$ with a multiplication $\cdot_{\Psi}$ and two operators $R_{\chi},S_{\chi}$ defined by
\begin{gather}
  (a,m)\cdot_{\Psi}(b,n)=(ab,an+mb+\Psi(a,b)),\;\forall a,b\in A,\,m,n\in M,\label{a} \\
  R_{\chi}(a,m)=(R(a),\chi_R(a)+R_M(m)),\;\forall a\in A,\,m\in M,\label{b} \\
  S_{\chi}(a,m)=(S(a),\chi_S(a)+S_M(m)),\;\forall a\in A,\,m\in M.\label{c}
\end{gather}
Moreover, we get an abelian extension
$$ \xymatrix@C=0.5cm{
    0 \ar[r] & (M,R_M,S_M) \xrightarrow{\binom{\,0\,}{\mathrm{Id}_M}} & (A\oplus M,R_{\chi},S_{\chi})  \xrightarrow{(\,\mathrm{Id}_A\quad 0\,)} & (A,R,S) \ar[r] & 0 }$$
which is easily seen to be isomorphic to the original one (\ref{ses}).
\begin{prop}
  $\mathrm{(i)}$ Different choices of the section $t$ give the same Rota-Baxter system bimodule structures on $(M,R_M,S_M)$;

  $\mathrm{(ii)}$ the cohomological class of $(\Psi, \chi_R, \chi_S)$ does not depend on the choice of sections.
\end{prop}

\begin{proof}
Let $t_1$ and $t_2$ be two distinct sections of $p$. We define $\gamma :A\to M$ by $\gamma (a)=t_1(a)-t_2(a)$.

Since the Rota-Baxter system $(M,R_M,S_M)$ satisfies $uv=0$ for all $u,v\in M$, we have
\[t_1(a)m=t_2(a)m+\gamma(a)m=t_2(a)m.\]
So different choices of the section $t$ give the same Rota-Baxter system bimodule structures on $(M,R_M,S_M)$.

Now, we show that the cohomological class of $(\Psi, \chi_R, \chi_S)$ does not depend on the choice of sections.
\begin{align*}
   \Psi_1(a\otimes b)={} & t_1(a)t_1(b)-t_1(ab) \\
   ={} & (t_2(a)+\gamma(a))(t_2(b)+\gamma(b))-(t_2(ab)+\gamma(ab))\\
   ={} & (t_2(a)t_2(b)-t_2(ab))+t_2(a)\gamma(b)+\gamma(a)t_2(b)-\gamma(ab)\\
   ={} & (t_2(a)t_2(b)-t_2(ab))+a\gamma(b)+\gamma(a)b-\gamma(ab)\\
   ={} & \Psi_2(a\otimes b)+\delta^1(\gamma)(a\otimes b),
 \end{align*}
 \begin{align*}
  {\chi_R}_1(a)={} & \hat{R}(t_1(a))-t_1(R(a))\\
  ={} & \hat{R}(t_2(a)+\gamma(a))-(t_2(R(a))+\gamma(R(a)))\\
   ={} & (\hat{R}(t_2(a))-t_2(R(a)))+\hat{R}(\gamma(a))-\gamma(R(a))\\
   ={} & {\chi_R}_2(a)+R_M(\gamma(a))-\gamma(R(a))\\
   ={} & {\chi_R}_2(a)-\Phi_R^1(\gamma)(a),
\end{align*}
\begin{align*}
  {\chi_S}_1(a)={} & \hat{S}(t_1(a))-t_1(S(a))\\
  ={} & \hat{S}(t_2(a)+\gamma(a))-(t_2(S(a))+\gamma(S(a)))\\
   ={} & (\hat{S}(t_2(a))-t_2(S(a)))+\hat{S}(\gamma(a))-\gamma(S(a))\\
   ={} & {\chi_S}_2(a)+S_M(\gamma(a))-\gamma(S(a))\\
   ={} & {\chi_S}_2(a)-\Phi_S^1(\gamma)(a).
\end{align*}
That is, $(\Psi_1, {\chi_R}_1, {\chi_S}_1)=(\Psi_2, {\chi_R}_2, {\chi_S}_2)+d^1(\gamma)$. Thus $(\Psi_1, {\chi_R}_1, {\chi_S}_1)$ and $(\Psi_2, {\chi_R}_2, {\chi_S}_2)$ form the same cohomological class in $\mathrm{H}_{\mathrm{RBS}}^2(A,M)$.
\end{proof}

\begin{prop}
Let $M$ be a vector space and $R_M,S_M\in \mathrm{End}_{\mathbb{K}}(M)$. Then $(M,R_M,S_M)$ is a Rota-Baxter system with trivial multiplication. Let $(A, R,S)$ be a Rota-Baxter system. Two isomorphic abelian extensions of Rota-Baxter system $(A, R,S)$ by $(M,R_M,S_M)$ give rise to the same cohomology class in $\mathrm{H}^2_{\mathrm{RBS}} (A, M)$.
\end{prop}

\begin{proof}
Suppose $(\hat{A}_1,\hat{R}_1,\hat{S}_1)$ and $(\hat{A}_2,\hat{R}_2,\hat{S}_2)$ are two abelian extensions of $(A,R,S)$ by $(M,R_M,S_M)$ as is given in (\ref{iso}). Let $t_1$ be a section of $(\hat{A}_1,\hat{R}_1,\hat{S}_1)$. As $p_2\circ \zeta=p_1$, we have
\[p_2\circ (\zeta\circ t_1)=p_1\circ t_1=Id_A.\]
Therefore, $\zeta\circ t_1$ is a section of $(\hat{A}_2,\hat{R}_2,\hat{S}_2)$. Denote $t_2=\zeta\circ t_1$. Since $\zeta$ is a homomorphism of Rota-Baxter systems such that $\zeta|_M=Id_M, \zeta(am)=\zeta(t_1(a)m)=t_2(a)m=am$, so $\zeta|_M: M\to M$ is compatible with the induced Rota-Baxter system bimodule structures. We have
\begin{align*}
   \Psi_2(a\otimes b)={} & t_2(a)t_2(b)-t_2(ab)\\
   ={} & \zeta(t_1(a))\zeta(t_1(b))-\zeta(t_1(ab)) \\
   ={} & \zeta(t_1(a)t_1(b)-t_1(ab)) \\
    ={} & \zeta(\Psi_1(a\otimes b))\\
   ={} & \Psi_1(a\otimes b),
 \end{align*}
 \begin{align*}
  {\chi_R}_2(a)={} & (\hat{R}_2(t_2(a))-t_2(R(a)))\\
  ={} & \hat{R}_2(\zeta(t_1(a)))-\zeta(t_1(R(a)))\\
  ={} & \zeta(\hat{R}_1(t_1(a))-t_1(R(a)))\\
  ={} & \zeta({\chi_R}_1(a))\\
   ={} & {\chi_R}_1(a),
\end{align*}
 \begin{align*}
  {\chi_S}_2(a)={} & (\hat{S}_2(t_2(a))-t_2(S(a)))\\
  ={} & \hat{S}_2(\zeta(t_1(a)))-\zeta(t_1(S(a)))\\
  ={} & \zeta(\hat{S}_1(t_1(a))-t_1(S(a)))\\
  ={} & \zeta({\chi_S}_1(a))\\
   ={} & {\chi_S}_1(a).
\end{align*}
Hence, two isomorphic abelian extensions give rise to the same cohomology class in $H^2_{RBS} (A, M)$.
\end{proof}

Now we consider the reverse direction. Let  $(M,R_M,S_M)$  be a Rota-Baxter system bimodule over Rota-Baxter system $(A, R,S)$, given three linear maps
$\Psi : A \otimes A \to M$ and $\chi_R,\,\chi_S  : A \to M$, one can define a multiplication $\cdot_{\Psi}$ and two operators $R_{\chi},S_{\chi}$ on $A\oplus M$ by Equations (\ref{a}) and (\ref{b})(\ref{c}). The following fact is important:

\begin{prop}
The quadruple $(A\oplus M, \cdot_{\Psi}, R_{\chi},S_{\chi})$  is a Rota-Baxter system if and only if $(\Psi, \chi_R, \chi_S)$ is a 2-cocycle in the cochain complex of the Rota-Baxter system $(A, R, S)$ with coefficients in the Rota-Baxter system bimodule $(M, R_M, S_M)$.  In this case, we obtain an abelian extension
$$ \xymatrix@C=0.5cm{
    0 \ar[r] & (M,R_M,S_M) \xrightarrow{\binom{\,0\,}{\mathrm{Id}_M}} & (A\oplus M,R_{\chi},S_{\chi})  \xrightarrow{(\,\mathrm{Id}_A\quad 0\,)} & (A,R,S) \ar[r] & 0 }$$
and the canonical section $t=\binom{\,\mathrm{Id}_A\,}{0}: (A,R,S)\to (A\oplus M,R_{\chi},S_{\chi})$ endows $M$ with the original Rota-Baxter system bimodule structure.
\end{prop}

\begin{proof}
If $(A\oplus M,\cdot_{\Psi},R_{\chi},S_{\chi})$ is a Rota-Baxter system, then the associativity of $\cdot_{\Psi}$ implies
\begin{equation}
 a\Psi(b\otimes c)-\Psi(ab\otimes c)+\Psi(a\otimes bc)-\Psi(a\otimes b)c =0,
\end{equation}
which means $\delta^2(\Psi)=0$ in $\mathrm{C}^\bullet (A,M)$. Since $(R_{\chi},S_{\chi})$ are Rota-Baxter system operators, for any $ a,b\in A,\,m,n\in M$, we have
\begin{align}
  R_{\chi}((a,m))\cdot_{\Psi}R_{\chi}((b,n))={} & R_{\chi}(R_{\chi}(a,m)\cdot_{\Psi} (b,n) + (a,m)\cdot_{\Psi} S_{\chi}(b,n)), \\
  S_{\chi}((a,m))\cdot_{\Psi}S_{\chi}((b,n))={} & S_{\chi}(R_{\chi}(a,m)\cdot_{\Psi} (b,n) + (a,m)\cdot_{\Psi} S_{\chi}(b,n)).
\end{align}
Then $\chi_R, \chi_S, \Psi$ satisfy the following equations:
\begin{align*}
  & R(a)\chi_R(b)+\chi_R(a)R(b)+\Psi(R(a)\otimes R(b))\\
 ={} & R_M(\chi_R(a)b) +R_M(a\chi_S(b))+\chi_R(R(a)b+aS(b))\\
 + & R_M\Bigl(\Psi(R(a)\otimes b+a\otimes S(b))\Bigr),
\end{align*}
\begin{align*}
   & S(a)\chi_S(b)+\chi_S(a)S(b)+\Psi(S(a)\otimes S(b)) \\
  ={} & S_M(\chi_R(a)b)+  S_M(a\chi_S(b))+\chi_S(R(a)b+aS(b))\\
  + & S_M\Bigl(\Psi(R(a)\otimes b+a\otimes S(b))\Bigr),
\end{align*}
That is,
\[\partial^1(\chi_R, \chi_S)+\Phi^2(\Psi)=0.\]
Hence, $(\Psi, \chi_R, \chi_S)$ is a 2-cocycle.

Conversely, if $(\Psi, \chi_R, \chi_S)$ is a 2-cocycle, one can easily check that $(A\oplus M, \cdot_{\Psi}, R_{\chi},S_{\chi})$  is a Rota-Baxter system.

The last statement is clear.
\end{proof}

Finally, we show the following result:
\begin{prop}
Two cohomologous 2-cocycles give rise to isomorphic abelian extensions.
\end{prop}

\begin{proof}
Given two 2-cocycles $(\Psi_1, {\chi_R}_1, {\chi_S}_1)$ and $(\Psi_2, {\chi_R}_2, {\chi_S}_2)$, we can construct two abelian extensions $(A\oplus M, \cdot_{\Psi_1}, R_{\chi_1},S_{\chi_1})$ and $(A\oplus M, \cdot_{\Psi_2}, R_{\chi_2},S_{\chi_2})$ via Equations (\ref{a}) and (\ref{b})(\ref{c}). If they represent the same cohomology class in $\mathrm{H}^2_{\mathrm{RBS}} (A, M)$, then there exists two linear maps $\gamma_0:\mathbb{K}\to M, \gamma_1:A\to M$ such that
\[(\Psi_1, {\chi_R}_1, {\chi_S}_1)=(\Psi_2, {\chi_R}_2, {\chi_S}_2)+(\delta^1(\gamma_1),-\Phi^1(\gamma_1)-\partial^0(\gamma_0,\gamma_0)).\]
Notice that $\partial^0(\gamma_0,\gamma_0)=\partial^0\circ\Phi^0(\gamma_0)=\Phi^1\circ\delta^0(\gamma_0)$. Define $\gamma:A\to M$ to be $\gamma_1+\delta^0(\gamma_0)$. Then $\gamma$ satisfies
\[(\Psi_1, {\chi_R}_1, {\chi_S}_1)=(\Psi_2, {\chi_R}_2, {\chi_S}_2)+(\delta^1(\gamma),-\Phi^1(\gamma)).\]
Define $\zeta:A\oplus M \to A\oplus M$ by
\[\zeta(a,m):=(a,-\gamma(a)+m).\]
Then $\zeta$ is an isomorphism of $(A\oplus M, \cdot_{\Psi_1}, R_{\chi_1},S_{\chi_1})$ and $(A\oplus M, \cdot_{\Psi_2}, R_{\chi_2},S_{\chi_2})$.
\end{proof}

\section*{Acknowledgements} This research is supported by NSFC (No.12031014). We are very grateful to Guodong Zhou for valuable discussions and for suggesting this topic.

\end{document}